\documentclass[12pt,twoside]{amsart}
\usepackage{amsmath}
\usepackage{amssymb}
\usepackage{amscd}
\usepackage{xypic}
\xyoption{all}
\title[de Rham Cohomology of Period Domains]{de Rham Cohomology of Period Domains}
\author{Mohammad Reza Rahmati}
\thanks{}
\address{ Abdus Salam School of Mathematical Sciences, GCU, Lahore, Pakistan
\hfill\break 
\hfill\break \\
\hfill\break }
\email{mrahmati@cimat.mx, rahmati@sms.edu.pk}

\newcommand{\comments}[1]{}

\def \H{{\mathcal H}}

\newtheorem{theorem}{Theorem}[section]
\newtheorem{proposition}[theorem]{Proposition}

\newtheorem{definition}[theorem]{Definition}



\begin{document}

\begin{abstract}
This is a review article discussing the de Rham cohomology of period domains of Hodge structures. We explain it as the de Rham cohomology of differentiable stacks as of a moduli space. We also discuss the cohomology of the partial toroidal compactification of these domains using known formulas on cohomology or Chow rings of toric structures. The text is expository and we have tried to connect some existing ideas that probably their relations not processed in the literature of Hodge theory. We state the significance of ideas as they naturally could be related, probably with not serious mathematical proof. The proofs stated in the text maybe expressed in a more serious context.  
\end{abstract}

\maketitle

%\noindent Version: 

\section*{Introduction}
%%%%%%%%%%%%%%%%%%%%%%%%%%%%%%%%%%%%%%%%%%%%%%%%%%%%%%%%%%%%%%%%%%%%

Stacks were invented by Grothendieck to provide a general framework to study local-global phenomena in mathematics. Later Deligne and Mumford introduced what is now called Deligne-Mumford stack. M. Artin generalized Deligne-Mumford work in which it became a tool in algebraic geometry, specially in the study of quotient spaces. Every scheme is a Deligne-Mumford (DM-)stack and every DM-stack is an Artin stack, all called algebraic stacks. Algebraic stacks are a new type of spaces in algebraic geometry, having more flexibility in constructions which are impossible in schemes. The notions of analytic, differentiable, and topological stacks was introduced analogously in the corresponding categories. In this text our stacks are differentiable stacks and we also naturally regard them as topological stacks. 

A basic example is when a topological group $G$ acts on a point $*$, trivially. The quotient stack $[*/G]$ of this action is called classifying stack of $G$, and is denoted $BG$. Note that $(BG)_{mod}$ is a point. However $BG$ is far from being a trivial object. More precisely, the map $* \to BG$ makes $*$ into a principal $G$-bundle over $BG$, and this  $G$-bundle is universal. That is for every topological space $T$, the equivalence classes of morphisms $T \to BG$ are in bijection with the isomorphism classes of principal $G$-bundles over $T$. In this situation if $G$ is discrete, the quotient map $* \to BG$ becomes the universal cover of $BG$. 

A groupoid is a category such that any morphism between two objects is invertible. The action of the groupoid $\mathcal{G}$ on $X$ is a morphism $\mathcal{G} \times X \to X$ with two arrows $\eta_x:e \times x \Rightarrow x, \ (g_1.g_2) \times x \Rightarrow g_1.(g_2.x)$. If a groupoid acts on a set $X$ one may form the action groupoid representing this groupoid action, by taking the objects to be the elements of $X$ and the morphisms to be elements $g \in G$ such that $g.x =y$ and compositions to come from the binary operation in $G$. More explicitly the action groupoid is the set $G \times X$ (often denoted $G \ltimes X$) such that the source and target maps are $s(g,x)=x, \ t(g,x)=gx$. The action $\rho$ is equivalently thought of as a functor $\rho:BG \to Sets$, from the group $G$ regarded as a one-object groupoid, denoted by $BG$. This functors sends the single object to the set $X$. Let $Set_*$ be the category of pointed sets and $Set_* \to Sets$ be the forgetful functor. We can think of this as a universal set-bundle. Then the action groupoid is the pullbak 

\begin{equation}
\begin{CD}
[X/G] @>>> Sets_* \\
@VVV    @VVV\\
BG @>>> Sets
\end{CD}
\end{equation}
 
We are specially interested to torus actions in its groupoid sense. A Deligne-Mumford (DM) torus is $T \times BG$, where $T$ is a torus and $G$ is a finite abelian group. A smooth toric Deligne-Mumford stack is a smooth separated DM-stack $X$ together with an open immerssion of a Deligne-Mumford torus $\imath:\mathcal{T} \hookrightarrow X$ with dense image such that the action of $\mathcal{T}$ on itself extends to an action $\mathcal{T} \times X \to X$. In this case a morphism is a morphism of stacks which extends a morphism of Deligne-Mumford tori. By the same way we define toric varieties by combinatorial data of their fan, a toric stack may also be defined by a stacky fan. We shall explain the toric stacks as locally quotient stacks via stacky fans, see \cite{FMN, GS, J}. The stacks in this text are differentiable stacks however we begin the discussion in the category of topological stacks for simplicity.

\section*{Quotient and Toric Stacks}

Every toplogical stack $\mathfrak{X}$ has an underlying topological space called coarse moduli space denoted $X_{mod}$. There is a natural functorial map $\pi_{mod}:\mathfrak{X} \to X_{mod}$ called the moduli map. Roughly $X_{mod}$ is the best approximation of $\mathfrak{X}$ by a topological space. Assume that $\mathfrak{X}=[X/G]$ is a quotient stack. Then there is a natural quotient map $q:X \to [X/G]$, and this maps makes $X$ a principal $G$-bundle over $[X/G]$. The usual quotient map we know from topology is the composition $\pi_{mod} \circ q$. 
We can think of $\mathfrak{X}$ as a topological space $X_{mod}$ which at every point $x$ is decorated with a topological group $I_x$. The group $I_x$ is called the stabilizer or inertia group at $x$. These inertia groups are interwined in an intericate way along $X_{mod}$. When $\mathfrak{X}$ is Deligne-Mumford, all $I_x$ are discrete. 
At every point $x \in \mathfrak{X}$ we have a pointed map 
$(BI_x,x) \to (\mathfrak{X},x)$, \cite{N}. An example is the shpere $S^2$ with an action of $\mathbb{Z}_n$ as rotations
fixing the north and south poles. the stack $[S^2/\mathbb{Z}_n]$ has an underlying space which is homeomorphic to a sphere, however $[S^2/\mathbb{Z}_n]$ remembers the stabilizers at the two fixed points, namely the north and the south poles. Thus $[S^2/\mathbb{Z}_n]$ is like $B\mathbb{Z}_n$ at the fixed points and in the remaining points is like the sphere, \cite{N}. 

A stacky fan is a pair $(\Sigma, \beta)$, where $\Sigma$ is a fan on a lattice $L$ and $\beta:L \to N$ is a homomorphism to a lattice $N$, so that $coker(\beta)$ is finite. A stacky fan gives rise to a toric stack as follows. Let $X_{\Sigma}$ be the toric variety associated to $\Sigma$. The map $\beta^*:N^* \to L^*$ induces a homomorphism of Tori, $T_{\beta}:T_L \to T_N$, by naturally identifying $\beta$ with the induced map on lattices of 1-parameter subgroups. Since $coker(\beta)$ is finite, $\beta^*$ is injective, so $T_{\beta}$ is surjective. Let $G_{\beta}=\ker T_{\beta}$. Note that $T_L$ is the torus on $X_{\Sigma}$, and $G_{\beta} \subset T_L$ is a subgroup. The action of $G_{\beta}$ on $X_{\Sigma}$ is induced by the homomorphism $G_{\beta} \to T_L$, \cite{GS}. 
 
\begin{definition} \cite{GS}
If $(\Sigma,\beta)$ is a stacky fan, we define the toric stack $\mathcal{X}_{\Sigma,\beta}$ to be $[\mathcal{X}/G_{\beta}]$, with the torus $T_N=T_L/G_{\beta}$. 
\end{definition}

\noindent
Every toric stack arises from a stacky fan, since every toric stack is of the form $[X/G]$, where $X$ is a toric variety and $G \subset T_0$ is a subgroup of its torus. 
Associated to $X$ is a fan $\Sigma$ on the lattice $L=Hom(\mathbb{G}_m,T_0)$. The surjection of tori $T \to T/G$ induces a homomorphism of lattices of 1-parameter subgroups, $\beta:L \to N:=Hom(\mathbb{G}_m,T/G)$. The dual homomorphism $\beta^*:N^* \to L^*$ is the induced homomorphism of characters. Since $T \to T/G$ is surjective, $\beta^*$ is injective, and the cokernel of $\beta$ is finite. Thus $(\Sigma,\beta)$ is a stacky fan and $[X/G]=\mathcal{X}_{\Sigma,\beta}$, \cite{GS}.

One may expect that the toric DM-stack can be represented by a larger $Z$ with larger group $G$. For instance, the classifying stack $B\mu_3=[*/\mu_3]$ is a toric DM stack by the stacky fan $(N=\mu_3,0,0)$. It is isomorphic to $[\mathbb{C}^{\times}/\mathbb{C}^{\times}]$ where $\mathbb{C}^{\times}$ acts by $(\lambda,x) \to \lambda^3x$. It is simply possible that a toric stack correspond to different stacky fans. Motivated by the study of gerbes one introduces the notion of extended stacky fan. An extended stacky fan is a triple $(N, \Sigma, \beta^e)$ where $N, \Sigma$ are the same as the stacky fan $\Sigma$, and $\beta^e:\mathbb{Z}^m \to N$ is determined by $b_1,...,b_n$ and additional elements $b_{n+1},...,b_m$. Then the toric DM-stack is defined by $[Z^e/G^e]$, where $Z^e=Z \times (\mathbb{C}^{\times})^{m-n}$, and $G^e$ acts on $Z^e$ through the homomorphism $\alpha^e:G^e \to (\mathbb{C}^{\times})^{m-n}$ determined by the extended stacky fan, \cite{J}. 

A coherent sheaf on a DM stack $[Z/G]$ is a $G$-equivariant sheaf on $Z$. A line bundle on $[Z/G]$ is given by a character $\chi$ of $G$. Let $b$ be a positive integer. We denote by $\sqrt[a]{L/X}$ the fiber product

\begin{equation}
\begin{CD}
\sqrt[a]{L/X} @>>> B\mathbb{C}^*\\
@VVV                    @VV \wedge bV\\
X @>>L>                 B\mathbb{C}^*
\end{CD}
\end{equation}

\noindent
where $\wedge b$ is the map that sends an invertible sheaf to its $b$-th power. More explicitly an object of $\sqrt[a_1]{L/X}$ over $f:S \to X$ is a couple $(M, \phi)$ where $M$ is an invertable sheaf on $S$ and $\phi:M^{\otimes b} \stackrel{\cong}{\rightarrow} f^*L$ is an isomorphism, \cite{FMN}. 

The structure morphism $X \to \bar{X}$ into its coarse moduli space factors canonically via the toric morphisms $X \to X^{\text{rig}} \to X^{\text{can}} \to \bar{X}$ where the first map is an abelian gerbe over $ X^{\text{rig}} $, and is a fiber product of roots of toric divisors as in (2), and the last  is the minimal orbifold having $\bar{X}$ as a coarse moduli space. there exists a unique $(a_1,...,a_n) \in \mathbb{N}_{>}^n$ such that 

\begin{equation} 
X^{\text{rig}} \cong \sqrt[a_1]{D_1/X} \times_X ... \times_X \sqrt[a_n]{D_n/X} 
\end{equation}

\noindent
where $D_i$ is the divisor corresponding to the ray $\rho_i$. For instance, the stack $\mathbb{P}(w) \cong [(\mathbb{C}^{n+1} \setminus 0) / \mathbb{C}^*]$ is a toric DM-stack with deligne-Mumford torus $[\mathbb{C}^{n+1} / \mathbb{C}^*]  \cong \mathbb{C}^{n} \times B\mu_d$. It is canonical if $gcd(w_0,..., \hat{w_i},...,w_n)=1$. It is an orbifold if $gcd(w_0,...,w_n)=1$, \cite{FMN}.

\section{Cohomology of Stacks}

The reference for this section is \cite{KB}. In this section we consider differentiable stacks. Although the stacks discussed in the previous sections were defined in the topological category, however we will just consider them as  differentiable stacks. Differentiable stacks are stacks over the category of differentiable manifolds. They are the stacks associated to Lie groupoids. A groupoid $X_1 \rightrightarrows X_0$ is a Lie groupoid if $X_0$ and $X_1$ are differentiable manifolds, structure maps are differentiable, the source and target maps are submersion. There is associated a simplicial nerve to a Lie groupoid namely

\begin{equation} 
X^{\bullet}:=\{\ X_p:= \overbrace{X_1 \times_{X_0} X_1 \times_{X_0} ... \times_{X_0} X_1}^\text{p \ times}\ \}_p \qquad \rightleftharpoons \qquad (X_1 \rightrightarrows X_0) 
\end{equation}

\noindent
Then we get an associated co-simplicial object 

\begin{equation}
\Omega^q(X_0) \to \Omega^q(X_1) \to \Omega^q(X_2) \to ... \ \ \ , \qquad \partial=\sum_{i=0}^p (-1)^i \partial_i^* 
\end{equation}

\noindent
the cohomology groups $H^k(X,\Omega^q)$ are called the Cech cohomology groups of the groupoid $X=[X_1 \rightrightarrows X_0]$, \cite{KB}.

\begin{definition}
Let $\mathcal{X}$ be a differentiable stack. Then 

\begin{equation}
H^k(\mathfrak{X},\Omega^q)=H^k(X_1 \rightrightarrows X_0, \Omega^q)
\end{equation}

\noindent
for any Lie groupoid $X \rightrightarrows X_0$ giving an atlas for $\mathfrak{X}$. In particular this defines 

\begin{equation}
\Gamma(\mathfrak{X}, \Omega^q)=H^0(\mathfrak{X}, \Omega^q) 
\end{equation}

\end{definition}

There are a number of well definedness conditions to be checked at this stage, such as the independence of the definition with respect to the choice of representative groupoid. We refer the reader to \cite{KB} for the details. The cohomology groups $H^k(X,\Omega^q)$ are called the Cech-de Rham cohomology groups of the groupoid $X=[X_1 \rightrightarrows X_0]$. The double complex $A^{pq}:=\Omega^q(X_p)$ is called the de Rham complex of the stack $\mathfrak{X}$, and its cohomologies are called de Rham cohomologies of $\mathfrak{X}$. The cohomology of the DM-stack $X$ is defined to be $H_{DR}^n(X)=H_{DR}^n(X_1 \rightrightarrows X_0)$, for any groupoid atlas $X_1 \rightrightarrows X_0$ of the stack $X$, \cite{KB}. 

If $G$ is a Lie group then $H^k(BG, \Omega^0)$, is the group cohomology of $G$ calculated with differentiable cochains. When the stack $X$ is a quotient stack there is a clear explanation of its cohomology as an equivariant cohomology. There is a well-known generalization of the de Rham complex to the equivariant case, namely the Cartan complex $\Omega_G^{\bullet}(X)$, defined by

\begin{equation}
\Omega_G^{\bullet}(X):= (\displaystyle{\bigoplus_{2k+i=n}(S^k\mathfrak{g}^{\vee} \otimes \Omega^i(X))^G}, \ d_{DR}-\imath)
\end{equation}

\noindent
where $S^{\bullet}\mathfrak{g}^{\vee}$ is the symmetric algebra on the dual of the Lie algebra of $G$. The group $G$ acts on $\mathfrak{g}$ by adjoint representation on $\mathfrak{g}^{\vee}$ and by pull back of differential forms on $\Omega^{\bullet}(X)$. The differential is $d_{DR}-\imath$ where $\imath$ is the tensor induced by the vector bundle homomorphism $\mathfrak{g}_X \to T_X$ coming from differentiating the action. If $G$ is compact, the augmentation is a quasi-isomorphism, i.e

\begin{equation}
H_G^i(X) \stackrel{\cong}{\longrightarrow}H^i(Tot \ \Omega_G^{\bullet}(X_{\bullet}))
\end{equation}

\noindent
for all $i$ and the groupoid $X_{\bullet}$, \cite{KB}.

\begin{proposition} \cite{KB}
If the Lie group $G$ is compact, there is a natural isomorphism 

\begin{equation}
H_G^i(X) \stackrel{\cong}{\longrightarrow}H_{DR}^i(G \times X \rightrightarrows X)= H_{DR}^i([X/G]) 
\end{equation}

\end{proposition}

\noindent
As a corollary $H_{DR}^*(BG)=(S^{2*}\mathfrak{g}^{\vee})^G$ for a compact lie group $G$. If $G$ is not compact, then $H_{DR}([X/G])$ is still equal to equivariant cohomology. This fact holds for equivariant cohomology in general. 

Similarly, every topological groupoid defines a simplicial nerve $X_{\bullet}$ which gives rise to the double complex $C_q(X_p)$ of simplices. The total homology of this double complex are called singular homology of $X_1 \rightrightarrows X_0$. A simple example is to consider the transformation groupoid $G \times X \rightrightarrows X$ for a discrete group $G$. In this case we have the degenerate spectral sequence 

\begin{equation}
E_{p,q}^2=H_q(G,H_p(X)) \Rightarrow H_{p+q}(G \times X \to X)
\end{equation}

\noindent
When $X$ is a point $H_p(G \times X \rightrightarrows X) =H_p(G,\mathbb{Z})$. In general there exists interpretation of the singular homology $H_*([X/G])$ in terms of the equivariant homology of $X$ when the Lie group $G$ acts continuously on $X$, exactly similar to de Rham cohomology case. There is also the dual notion of singular cohomologies of the stack $X$ by replacing the double complex $C_q(X_p)$ with its dual $Hom(C_q(X_p),\mathbb{Z})$. Directly we obtain a pairing

\begin{equation}
H_k(X,\mathbb{Z}) \times H^k(X,\mathbb{Z}) \to \mathbb{Z}
\end{equation}

For example, the stack of Elliptic curves $M_{1,1}$ may be represented by the action of $Sl_2(\mathbb{Z})$, by the linear fractional transformations on the upper half plane in $\mathbb{C}$. Thus the homology of the stack of Elliptic curves is equal to the homology of $Sl_2(\mathbb{Z})$, \cite{KB}.  

\begin{theorem} \cite{KB}
Let $\mathfrak{X}$ be a topological Deligne-Mumford stack with coarse moduli space $\bar{X}$. Then the canonical morphism $X \to \bar{X}$ induces isomorphisms on $\mathbb{Q}$-valued cohomologies,

\[ H^k(\mathfrak{X},\mathbb{Q}) \stackrel{\cong}{\longrightarrow} H^k(X,\mathbb{Q}) \]

\end{theorem}

The theorem follows by taking an open cover of $\mathfrak{X}$ by quotient stacks and using the Cech spectral sequence. We know $H^*(BGl_n,\mathbb{Z})=\mathbb{Z}[t_1,...,t_n]$. $t_i \in \H^{2i}(BGl_n,\mathbb{Z})$ is called the universal chern class. Given a rank $n$-vector bundle $E$ over a stack $X$, we get an associated morphism of stacks $f:X \to BGl_n$, such that the following diagram is commutative;

\begin{equation}
\begin{CD}
B @>>> * \\
@VVV @VVV\\
X @>f>> BGl_n  
\end{CD} \qquad \leftrightarrows \qquad \begin{CD}
E @>>> [\mathbb{C}^n/Gl_n] \\
@VVV @VVV\\
X @>f>> BGl_n  
\end{CD}
\end{equation}

\noindent
$B$ is the principal $Gl_n$-bundle of frames of $E$. The $i$-th chern class of the vector bundle $E$ is defined by $c_i(E) :=f^* \ t_i$, \cite{KB}. 

\begin{theorem} \cite{KB} 
If all the odd degree cohomologies of the stack $X$ vanishes then $H^*(E \setminus X)=H^*(X)/c_n(E)$. 
\end{theorem}

\section{Period and Mumford-Tate domains}

We briefly mention some basic definitions from \cite{KP} on period domains. Let $V$ be finite dimensional $\mathbb{Q}$-vector space, and $Q$ a non-degenerate bilinear map $Q:V \otimes V \to \mathbb{Q}$ which is $(-1)^n$-symmetric for some fixed $n$. A Hodge structure is given by a representation  

\begin{equation}
\phi:\mathbb{U}(\mathbb{R}) \to Aut(V, Q)_{\mathbb{R}}, \qquad \mathbb{U}(\mathbb{R}) =\left( 
\begin{array}{cc}
a  &  -b\\
b &    a
\end{array} \right), \ a^2+b^2 =1
\end{equation}

\noindent
It decomposes over $\mathbb{C}$ into eigenspaces $V^{p,q}$ such that $\phi(t).u= t^p\bar{t}^q.u$ for $u \in V^{p,q}$, and $\overline{V^{p,q}}=V^{q,p}$. It is well-known that $Ad(\phi):\mathbb{U}(\mathbb{R}) \to Aut(\mathfrak{g}_{\mathbb{R}},B)$ defines a weight $0$ Hodge structure, where $B$ is the Killing form. Then $\mathfrak{g}_{\mathbb{C}}$ has a decomposition 

\begin{equation}
\mathfrak{g}=\mathfrak{g}_{\phi}^- \oplus \overline{\mathfrak{g}_{\phi}^-} \oplus \mathfrak{h}_{\phi}, \qquad \mathfrak{g}^-:=\oplus_{i >0}\mathfrak{g}_{\phi}^{-i,i} 
\end{equation}

\noindent
and $T_{\phi,\mathbb{R}}D \otimes \mathbb{C}= \mathfrak{g}_{\phi}^- \oplus \overline{\mathfrak{g}_{\phi}^-}, \ T_{\phi}D=\mathfrak{g}_{\phi}^-$. The Mumford-Tate group $M_{\Phi}$ of the Hodge structure $\phi$ is the smallest $\mathbb{Q}$-algebraic subgroup containing the image of $\phi$, and is a reductive Lie group, that satisfies the almost product decomposition $M_{\Phi}=M_1 \times ... \times M_l \times T$. The associated Mumford-Tate domains $D_i$ (the $M_i(\mathbb{R})$-orbit of $\phi(s_0)$) have a splitting of the 

\begin{equation}
D_{M_{\Phi}}=D_1 \times ... \times D_l
\end{equation}

\noindent
Let $T_1,...,T_n$ be generators of the monodromy group $\Gamma$. Then a partial compactification $\Gamma \setminus D_{\sigma}$ can be defined  by the cone 

\begin{equation}
\sigma=\displaystyle{\sum_{j=1}^n \mathbb{R}_{\geq 0} N_j, \qquad N_j=\log T_j} 
\end{equation}

\noindent
A boundary point is a nilpotent orbit associated to a face of $\sigma$. Set 

\begin{equation}
\Gamma(\sigma)^{\text{gp}}=\exp(\sigma_{\mathbb{R}}) \cap G_{\mathbb{Z}}, \ \  \Gamma(\sigma)=\exp(\sigma) \cap G_{\mathbb{Z}}
\end{equation}

\noindent
The monoid $\Gamma (\sigma)$ defines the toric variety 

\begin{equation} 
D_{\sigma}:=Spec([\mathbb{C}[\Gamma(\sigma)^{\vee}])_{an} \cong Hom(\Gamma(\sigma)^{\vee},\mathbb{C}) 
\end{equation}

\noindent
with the torus

\begin{equation}
T_{\sigma}:=Spec(\mathbb{C}[\Gamma(\sigma)^{\vee \text{gp}}])_{an} \cong Hom(\Gamma(\sigma)^{\vee \text{gp}} , \mathbb{G}_m) \cong \mathbb{G}_m \otimes \Gamma(\sigma)^{\text{gp}} 
\end{equation}

\noindent
We assume that $\Gamma$ is strongly compatible with $\Sigma$,  

\begin{itemize}
\item $Ad(\gamma).\sigma \in \Sigma, ( \forall \gamma \in \Gamma, \sigma \in \Sigma )$
\item $\sigma_{\mathbb{R}}=\mathbb{R}_{\geq 0}\langle \log \Gamma(\sigma) \rangle$.
\end{itemize}

Let $\phi: \mathbb{U} \to M$ be a Hodge structure, with the associated Mumford-Tate domain $D_M=M(\mathbb{R}). \phi$. Set $\mathfrak{m}=Lie(M)$. The boundary domponent associated to $ \mathbb{Q}_{\geq 0} \langle N_1,...,N_r \rangle \subset \mathfrak{m}$ is 

\begin{equation}
B_{\sigma}:=\tilde{B}_{\sigma}/e^{{\langle \sigma \rangle}_{\mathbb{C}}}
\end{equation}

\noindent
where

\begin{equation}
 \tilde{B}_{\sigma}:=\{ F^{\bullet} \in \check{D} \ | \ Ad (e^{ \sigma}).F^{\bullet} \ \text{is a nilpotent orbit} \}
\end{equation}

\noindent  
Kato-Usui define a generalization of the Hodge domains as follow, 

\begin{equation} 
D_{M,\sigma} :=\displaystyle{\amalg_{\sigma \in \Sigma} \ \{\ Z \subset \check{D_M}\ |\ \ Z \ \ \text{is a} \  \sigma-\text{nilpotent orbit} \}=\amalg_{\sigma \in \Sigma}\ B(\sigma)} 
\end{equation}

\noindent
This always contain $B(\{0\})=D_M$. It gives a toric variety structure for $D_M$, denoted $D_{M,\sigma}$. In particular $D_{M,\sigma}=D_{M,\text{faces of}\ \sigma}$. When $\Gamma \subset M(\mathbb{Z})$ is a neat subgroup of finite index , then there exist successive covers 

\begin{equation}
\overline{B(N)} \twoheadrightarrow ... \overline{B(N)}_{(k)} \twoheadrightarrow ... \twoheadrightarrow \overline{B(N)}_{(1)} \twoheadrightarrow \overline{D(N)} 
\end{equation}

\noindent
with $k>1$, and intermediate Jacobian fibers at each stage, and also $\overline{D(N)}$ being discrete, \cite{KP}. 

We now come back to the discussion of the toric stacks. We claim that $\overline{D_M}$ with the above structure is a toric stack. The injective map $\beta : \gamma \mapsto e^{\gamma} \mapsto \log \circ e^{\gamma}$ restricted to $\Sigma=\mathbb{Z}\langle N_1,...,N_r \rangle \to \mathbb{Z}^r$ defines a homomorphism of stacky fan type, i.e it has finite cokernel. We will consider $\beta$ as an isomorphism onto its image $\mathfrak{m}_{\mathbb{Z}}$ a maximal lattice in $\mathfrak{m}$. As we explained this establishes a toric stack structure on $\overline{D_M}$. The torus action is explained via the partial toric structure on the stratification of $\overline{D_M}$ as explained through the whole process in this section. One should note that this does not mean $\overline{D_M}$ is a toric variety, however on each open strata we have a torus action. One expects that the torus action can never extend to the whole $\overline{D_M}$ (nor either $D_M$ for basic reasons).

\section{Characteristic cohomology}

In a general situation when $X$ is a complex manifold and $W \subset X$ a holomorphic sub-bundle. Then $I=W^{\perp} \subset T^*(X)$ is also holomorphic and let $\bar{I}$ be its conjugate. Let $\mathcal{I}^{\bullet \bullet} \subset \mathcal{A}^{\bullet \bullet}$ be the differential ideal of sections of $I \oplus \bar{I}$ in the algebra of smooth differential forms on $X$. The characteristic cohomology of $X$ denoted $H^*_{\mathcal{I}}(X)$ is the cohomology of the double complex $ \subset \mathcal{A}^{\bullet \bullet}/\mathcal{I}^{\bullet \bullet}$, \cite{GGK}. 

Let $D$ be a period domain for a PHS $(V,Q,\phi)$ of weight $n$ and set $\Gamma_{\mathbb{Z}}=Aut(V_{\mathbb{Z}},\mathbb{Q})$. In the tangent bundle $TD$ there is a homogeneous sub-bundle $W$ whose fiber at $\phi$ is 

\begin{equation}
W_{\phi}=\mathfrak{g}_{\phi}^{-1,1}=\{\psi \in T_{\phi}D:\psi(F_{\phi}^p) \subset F_{\phi}^{p+1} \}
\end{equation}

\noindent
defined by, namely infinitesimal period relations (IPR). There is a natural inclusion $T_{F^{\bullet}}\check{D} \subset \oplus_p Hom(F^p,V_{\mathbb{C}}/F^p)$. We shall assume that, the canonical sub-bundle $W \subset T\check{D}$ given by the infinitesimal period relation is defined by 

\begin{equation}
W_{F^{\bullet}}=T_{F^{\bullet}} \check{D} \cap (\oplus_p Hom(F^p,F^{p-1}/F^p)) 
\end{equation}

\noindent
The bundle $W \to \check{D}$ is acted by $G_{\mathbb{C}}$, and the action of $G_{\mathbb{R}}$ on $W \to D$ leaves invariant the metric given by Cartan killing form at each point. With the identification $T_{\phi}D=\oplus_{i >0}\mathfrak{g}^{-i,i}$ we have $W_{\phi}=\mathfrak{g}_{\phi}^{-1,1}$. In our case, take $X=D_M \subset D$ to be a Mumford-Tate domain, and $W_M \subset TD_M$ is the infinitesimal period relation. Denote by $\Lambda_M^{\bullet \bullet}$ the complex of $G(\mathbb{R})$-invariant forms in $A^{\bullet \bullet}(D_M)/\mathcal{I}^{\bullet \bullet}$ with the operator $\delta:\Lambda_M^{\bullet} \to \Lambda^{\bullet +1}$.  $H^*(\Lambda_M^{\bullet}, \delta_M)$ is called the universal characteristic cohomology, \cite{GGK}. 

\begin{proposition} \cite{GGK}
The universal characteristic cohomology i.e the cohomology of $A^{\bullet \bullet}(D_M)/\mathcal{I}^{\bullet \bullet}, \delta:\Lambda_M^{\bullet} \to \Lambda^{\bullet +1}):=H^*(\Lambda_M^{\bullet}, \delta_M)$ is equal to the equivariant cohomology $H^{*}(I^{\bullet \bullet})$. Moreover, $H^{2p-1}(\Lambda_M^{\bullet}, \delta_M)=0$ and 

\begin{equation}
H^{2p}(\Lambda_M^{\bullet}, \delta_M)=(\Lambda^{p,p}){^{\mathfrak{m}^{0,0}}}.
\end{equation}
 
\end{proposition}

The universal characteristic cohomology is equal to the de Rham cohomology of the stack $D_M$. In case $D_M=D$ and $M=G$, the universal characteristic cohomology is generated by the chern forms of Hodge bundles. According to the proposition 3.2 and the remark after that, the universal characteristic cohomology can be understood as the stack cohomology of the Mumford-Tate domains. One can ask; hat are the conditions that the chern forms of Hodge bundles on $D_M$ generate the characteristic cohomology? In case this holds what are the relations?.  

Keeping the notation as section (0), let $P \to B$ be a principal $(\mathbb{C}^{\times})^m$-bundle. The group $G^e$ acts on the fiber product $P \times_{(\mathbb{C}^{\times})^m} Z^e$ via the map $\alpha^e$, where $\alpha^e$ is given by the stacky fan. Each $\theta \in M=N^*$, gives $\chi_{\theta}:(\mathbb{C}^{\times})^m \to (\mathbb{C}^{\times})$ and assume the line bundle $\psi_{\theta}=P \times_{\chi_{\theta}} \mathbb{C}$ be given by the ray $\theta$. Then according to the main result of \cite{J} one knows that

\begin{equation}
H_{orb}^*([(P \times_{\alpha^e} Z^e)/G^e]=\dfrac{H^*(B)[N] \otimes \mathbb{Q}[N]^{\Sigma^e}}{<c_1(\psi_{\theta})+\theta(b_i)y^{b_i}>_{\theta \in M}}
\end{equation}

\vspace{0.2cm}

\noindent
where $H_{orb}^*$ is the orbifold Cohomology ring, $H^*$ the usual cohomology ring, $<.>$ means the ideal generated by $.$, and $\mathbb{Q}[N]^{\Sigma^e}$ is $\oplus_{c \in N}y^c$ where $y$ shows formal variables and the multiplication $y^c.y^{c'}=y^{c+c'}$ when there is $\sigma \in \Sigma $ such that $\beta^e(c)$ and $\beta^e(c') \in \sigma$, and $0$ otherwise, \cite{J}. We apply the above to the homogeneous manifold $G/P$ of the principal bundle $P \hookrightarrow G \to G/P=D$ and the homogeneous vector bundle $\Lambda$ of characteristic cohomology. Then characteristic cohomology finds the form 

\begin{equation}
H^*(\Lambda_M^{\bullet}, \delta_M)=\dfrac{H^*(D_M)[\mathfrak{m}]^{\Sigma}}{<c_1(\psi_{\lambda})+\lambda(b_i)y^{b_i}>_{\lambda}}
\end{equation}

\noindent
where we are using notations in (28). 
The proof follows from the fact that $TD_M$ is an equivariant $\mathbb{C}^*$-bundle where $\mathbb{C}^*$ acting as the adjoint action of the maximal torus, or equivalently through a 1-PS with weights equal to Cartan integers. in fact one has the simple relation $\Lambda^{p,q}=\bigwedge^{p} \mathfrak{m}^{-1,1} \otimes \bigwedge^{q} \mathfrak{m}^{1,-1}$. This provides a toric stack bundle which the proposition (28) applies.

\end{document}